\documentclass[11pt]{article}

\usepackage{amssymb, amsmath, amsthm, mathtools, comment}

\newcommand{\dd}{\mathrm{d}}
\newcommand{\E}{\mathbf{E}}
\newcommand{\p}{\mathbf{P}}
\newcommand{\R}{\mathbb{R}}
\newcommand{\N}{\mathbb{N}}
\newcommand{\fl}{\overline{f}}

\newcommand{\ind}{\mathbb{I}}
\newcommand{\Ssec}[2]{\genfrac\{\}{0pt}{0}{#1}{#2}}
\newcommand{\Sfirst}[2]{\genfrac[]{0pt}{0}{#1}{#2}}

\theoremstyle{plain}
\newtheorem{lemma}{Lemma}
\newtheorem{theorem}{Theorem}

\theoremstyle{remark}
\newtheorem{example}{Example}
\newtheorem{remark}{Remark}

\title{Nearly critical Galton--Watson processes}
\author{P\'eter Kevei\footnote{Bolyai Institute, University of Szeged,
Hungary.
E-mail: \texttt{kevei@math.u-szeged.hu}} \and 
Kata Kubatovics\footnote{Bolyai Institute, University of Szeged, Hungary.
E-mail: \texttt{kubatovicskata@gmail.com}}}

\date{}

\begin{document}

\maketitle

\begin{abstract}
We investigate Galton--Watson processes in varying environment, for which
$\fl_n \uparrow 1$ and $\sum_{n=1}^\infty (1-\fl_n) = \infty$, where 
$\fl_n$ stands for the offspring mean in generation $n$. 
Since the process dies out almost surely, to obtain nontrivial limit 
we consider two scenarios: conditioning on non-extinction, or adding 
immigration. In both cases we show that the 
process converges in distribution without normalization to a nondegenerate 
compound-Poisson limit law. The proofs rely on the shape 
function technique, worked out by Kersting \cite{Kersting}.


\end{abstract}

\section{Introduction}

A Galton--Watson branching process in varying environment
(BPVE) is defined as
$X_0=1$, and 
\begin{equation} \label{eq:X}
X_n = \sum_{j=1}^{X_{n-1}} \xi_{n,j}, \quad n\in\N = \{ 1, 2, \ldots \},
\end{equation}
where $\{\xi_{n,j}\}_{n, j \in \N}$ are nonnegative independent random variables,
such that for each $n$,  
$\{\xi_n, \xi_{n,j}\}_{j \in \N}$ are identically distributed. 
We can interpret $X_n$ as 
the size of the $n$th generation of a population and $\xi_{n,j}$ represents 
the number of offsprings produced by the $j$th individual in generation $n-1$.
These processes are natural extensions of usual homogeneous Galton--Watson
processes, where the offspring distribution does not change with the 
generation. 

The investigation of BPVE started 
in the early 1970's with the work of Church \cite{Church} and 
Fearn \cite{Fearn}.
There is a recent increasing interest again in these processes,
which was triggered by Kersting \cite{Kersting}, who obtained
necessary and sufficient condition for almost sure extinction
for {regular processes}. In \cite{Kersting}
he introduced the {shape function} of a generating function (g.f.),
which turned out to be the appropriate tool for the analysis. 
Moreover, in \cite{Kersting} Yaglom-type limit theorems (that is 
conditioned on non-extinction) was obtained.
Under different regularity conditions Bhattacharya and Perlman 
\cite{Bhatt} also proved Yaglom-type limit theorems both in the 
discrete and in the continuous time setting, extending the results
of Jagers \cite{Jagers}. For multitype processes
these questions were investigated by Dolgopyat et al.~\cite{DHKP}.
Cardona-Tob\'on and Palau \cite{CardonaPalau} obtained probabilistic 
proofs of the results in \cite{Kersting} using spine decomposition.
On general theory of BPVE we refer to the 
recent monograph by Kersting and Vatutin \cite{KerVat}, and to the 
references therein.

Here we are interested in \emph{nearly critical Galton--Watson processes}.
Let $f_n(s) = \E  s^{\xi_n}$, $s \in [0,1]$, denote the g.f.~of the offspring 
distribution of the $n$th generation, and  put
$\fl_n \coloneqq f_n'(1) = \E \xi_n$ for the offspring mean. We assume the following
conditions:
\begin{itemize}
\item[(C1)] $\fl_n<1$, $\lim_{n\to\infty}\fl_n = 1$, 
$\sum_{n=1}^{\infty}(1-\fl_n) = \infty$,
\item[(C2)] $\lim_{n\to\infty} \frac{f_n''(1)}{1-\fl_n} = \nu \in [0,\infty)$,
\item[(C3)] $\lim_{n\to \infty}\frac{f_n'''(1)}{1- \fl_n} = 0$, if $\nu>0$.
\end{itemize}
Condition (C1) means that the process is nearly critical and the 
convergence $\fl_n \uparrow 1$ is not too fast.
Consequently, $\E X_n = \prod_{j=1}^n \fl_j = \fl_{0,n}
\to 0$, thus the process dies out almost surely. 
Condition (C2) implies that $f_n''(1) \to 0$, thus
$f_n(s) \to s$, so the branching mechanism converges to the 
degenerate branching. 
Therefore, this nearly critical model does not have 
a natural homogeneous counterpart.
To obtain nontrivial limit theorems we can condition on non-extinction,
or add immigration. 

Conditioned on $X_n > 0$, in our Theorem \ref{thm:yaglom-geom} 
we prove that $X_n$ converges 
in distribution without normalization to a geometric distribution.
Although, the process is nearly critical, 
in this sense its behavior is similar to a homogeneous subcritical
process, where no normalization is needed to obtain a limit, see 
Yaglom's theorem \cite{Yaglom} (or in 
\cite[Theorem I.8.1, Corollary I.8.1]{AthreyaNey}).
For a regular BPVE (in the sense of \cite{Kersting}),
necessary and sufficient conditions were 
obtained for the tightness of $X_n$ (without normalization) 
conditioned on $X_n > 0$ in \cite[Corollary 2]{Kersting}.
However, to the best of our knowledge, our result is the first 
proper limit theorem without normalization. 
Similarly to Yaglom's theorem in the homogeneous critical case
(see \cite{Yaglom} or \cite[Theorem I.9.2]{AthreyaNey}), exponential limit 
for the properly normalized process conditioned on 
non-extinction was obtained in several papers, see 
e.g.~\cite{Jagers}, \cite{Bhatt}, \cite{Kersting}.
\smallskip

Adding immigration usually leads to similar behavior as conditioning 
on non-extinction. Let $Y_0=0$ and 
\begin{equation} \label{eq:Y}
Y_n=\sum_{j=1}^{Y_{n-1}}\xi_{n,j}+\varepsilon_n, \quad n\in\N,
\end{equation}
where $\{\xi_{n,j}, \varepsilon_n\}_{n, j \in \N}$ are nonnegative, 
independent random variables such that $\{\xi_n, \xi_{n,j}\}_{j \in \N}$ are
identically distributed. As before $\xi_{n,j}$ 
is the number of offsprings of the $j$th individual 
in generation $n-1$ and $\varepsilon_n$ is the number of immigrants.

The study of nearly critical Galton--Watson processes with immigration 
was initiated by Gy\"orfi et al.~\cite{GyIPV}, where it was assumed that the 
offsprings have Bernoulli distribution. In this case the resulting process
is an inhomogeneous first-order integer-valued autoregressive process (INAR(1)).
INAR processes are important in various fields of applied probability, for 
theory and applications we refer to the survey of Wei\ss \cite{Weiss}.
The setup of \cite{GyIPV} was extended by Kevei \cite{Kevei}, allowing 
more general offspring distributions. The multitype case was studied 
by Gy\"orfi et al.~\cite{GyIKP}. 

In general, Galton--Watson processes with 
immigration in varying environment (that is, with time-dependent immigration)
are less studied. Gao and Zhang \cite{GaoZhang} proved central limit theorem
and law of iterated logarithm. Isp\'any \cite{Ispany} obtained diffusion
approximation in the strongly critical case, i.e.~when instead of (C1) condition
$\sum_{n=1}^\infty |1 - \fl_n| < \infty$ holds. Gonz\'alez et al.~\cite{GKMP}
investigated  a.s.~extinction, and obtained limit theorems for the properly
normalized process. Since we prove limit theorems for the process without 
normalization, both our results and assumptions are rather different 
from those in the mentioned papers.

Generalizing the results in \cite{Kevei},
in Theorems \ref{thm:bev1} and \ref{thm:bev2} we show that under
appropriate assumption on the immigration, the slow extinction trend 
and the slow immigration trend are balanced and we get a nontrivial 
compound-Poisson limit distribution without normalization. 

\smallskip

The rest of the paper is organized as follows. Section 
\ref{sec:results} contains the  main results.
All the proofs are gathered together in Section \ref{sec:proofs}.
The proofs are rather technical, and based on analysis of the 
composite g.f.
We rely on the shape function technique worked out in \cite{Kersting}.

\section{Main results} \label{sec:results}

\subsection{Yaglom-type limit theorems} \label{ssec:Yaglom}

Consider the BPVE $(X_n)$ in \eqref{eq:X}. Condition (C1) implies 
that the process dies out almost surely. 
In \cite{Kersting} a BPVE process  $(X_n)$ is called \emph{regular} if
\begin{equation} \label{eq:regular}
\E ( X_n^2 \ind(X_n \geq 2)) \leq c \E (X_n \ind(X_n \geq 2)) \, 
\E [ X_n | X_n \geq 1]
\end{equation}
holds for some $c > 0$ for all $n$, where $\ind$ stands for the indicator
function.  In our setup, under (C1)--(C3) if 
$\nu = 0$ one can construct an example for which \eqref{eq:regular} does not 
hold. However, if $\nu > 0$ then
(C1) and (C2) imply \eqref{eq:regular}, in
which case the results in \cite{Kersting} apply. In particular, 
\cite[Corollary 2]{Kersting} states that the sequence 
$\mathcal{L}(X_n | X_n > 0)$ is tight if and only if 
$\sup_{n \geq 0} \E [ X_n | X_n > 0] < \infty$. 
Here $\mathcal{L}(X_n|X_n>0)$ stands for the law of $X_n$ conditioned on
non-extinction. In Lemma \ref{lemma:f-conv} below
we show that the latter condition holds, in fact the limit exists. Therefore,
the sequence of conditional laws is tight. In the next result, we prove 
that the limit distribution also exists.

The random variable $V$ has geometric distribution with parameter $p \in (0,1]$,
$V \sim $ Geom$(p)$,
if $\p ( V = k) = (1-p)^{k-1} p$, $k=1,2,\ldots$. 

\begin{theorem}\label{thm:yaglom-geom}
Assume that (C1)--(C3) are satisfied. Then, for the BPVE $(X_n)$ in \eqref{eq:X}
\[
\mathcal{L}(X_n|X_n>0)\overset{\mathcal{D}}{\longrightarrow} 
\textup{Geom}\left(\frac{2}{2+\nu}\right) \quad \text{as $n\to\infty$},
\]
where 
$\overset{\mathcal{D}}{\longrightarrow}$ denotes convergence in 
distribution.
\end{theorem}

The result holds with $\nu = 0$, in which case the limit is degenerate at 1.

\begin{example}
Let $f_n(s) = f_n[0] + f_n[1] s + f_n[2] s^2$, 
then
$\fl_n = f_n[1] + 2 f_n[2]$, 
$f_n''(1) = 2 f_n[2]$, and $f_n'''(1) = 0$.
Assuming that $f_n[0]+f_n[2] \to 0$, $f_n[0]> f_n[2]$,
$\sum_{n=1}^\infty (f_n[0] - f_n[2]) = \infty$
and $2 f_n[2] / ( f_n[0]- f_n[2])  \to \nu \in [0,\infty)$ as $n\to\infty$, the conditions of 
Theorem \ref{thm:yaglom-geom} are fulfilled.
\end{example}

\subsection{Nearly critical branching processes with immigration}
\label{ssec:bevandorlas}

Recall $(Y_n)$ the BPVE with immigration from \eqref{eq:Y}, and 
introduce the factorial moments of the immigration 
\begin{equation*} \label{eq:def-mnk}
m_{n,k} \coloneqq \E [\varepsilon_n(\varepsilon_n-1) \ldots (\varepsilon_n-k+1)],
\quad k\in \N.
\end{equation*}

\begin{theorem}\label{thm:bev1}
Assume that (C1)--(C3) are satisfied and
\begin{itemize}
\item[(C4)] $\lim_{n\to\infty}\frac{m_{n,k}}{k!(1-\fl_n)} = \lambda_k, \quad k = 
1,2,\ldots,K$ and $\lambda_K=0$.
\end{itemize}
Then, for the BPVE with immigration $(Y_n)$ in \eqref{eq:Y}
\[
Y_n\overset{\mathcal{D}}{\longrightarrow} Y \quad \text{as  $n\to\infty$},
\]
where the random variable $Y$ has compound-Poisson distribution, and its g.f.
\[
f_Y(s) 
= \exp
\left\{ -\sum_{k=1}^{K-1} 
\frac{2^k\lambda_k}{\nu^k} \left( \log\left( 
1+\frac{\nu}{2}(1-s)\right) +
\sum_{i=1}^{k-1}(-1)^{i}\frac{\nu^i}{i2^i}
(1-s)^i\right)\right\}.
\]
\end{theorem}

Under further assumptions we might allow infinitely many nonzero $\lambda$'s.

\begin{theorem}\label{thm:bev2}
Assume that (C1)--(C3) are satisfied and
\begin{itemize}
\item[(C4')] $\lim_{n\to\infty}\frac{m_{n,k}}{k!(1-\fl_n)} = \lambda_k$,
$k=1,2,\ldots$, and
$\limsup_{n \to \infty} \lambda_n^{1/n} \leq 1$.
\end{itemize}
Then, for the BPVE with immigration $(Y_n)$ in \eqref{eq:Y}
\[
Y_n \overset{\mathcal{D}}{\longrightarrow} Y \quad \text{as $n\to\infty$},
\]
where the random variable $Y$ has compound-Poisson distribution, and its g.f.
\begin{equation} \label{eq:Y-genfc}
f_Y(s) = 
\exp \left\{ 
-\sum_{k=1}^{\infty}\frac{2^k\lambda_k}{\nu^k}
\left(\log\left(1+\frac{\nu}{2}(1-s)\right) 
+ \sum_{i=1}^{k-1} (-1)^{i}\frac{\nu^i}{i2^i}
(1-s)^i\right)\right\}.
\end{equation}
\end{theorem}

\begin{remark}
If $K=2$ in Theorem \ref{thm:bev1} then $Y$ has negative binomial 
distribution with parameters 
$\frac{2\lambda_1}{\nu}$ and $\frac{2}{2+\nu}$ as shown in Theorem 5 of \cite{Kevei}.
In particular, Theorems \ref{thm:bev1} and \ref{thm:bev2} are far-reaching
generalizations of \cite[Theorem 5]{Kevei}.
\end{remark}

\begin{remark}
If $\nu = 0$ in condition (C2) in Theorem \ref{thm:yaglom-geom},
then the results in Theorems \ref{thm:bev1} and \ref{thm:bev2} follows 
from \cite[Theorem 4]{Kevei}.
\end{remark}

\begin{remark}
Conditions (C4) and (C4') above are not comparable in 
the sense that none of them implies the other. Indeed, if (C4) holds 
with $K = 3$, then it is possible that 4th moments does not exists,
while in (C4') all moments have to be finite.

One can construct examples for which $\lambda_2 = 0$, $\lambda_3=1$,
$\lambda_4 = \infty$. In this case (C4) in Theorem \ref{thm:bev1} holds
with $K= 2$. On the other hand, it is easy to show that if 
$\lambda_{n_1} = \lambda_{n_2} = 0$, for some $n_1 < n_2$, 
then $\lambda_n = 0$ for $n_1 < n < n_2$.
\end{remark}

The g.f.~$f_Y$ of the limit has a rather complicated form.
In the proof, showing the pointwise convergence of the g.f.s, we 
prove that the accompanying laws converge in distribution to $Y$.
Since the accompanying laws are compound-Poisson, this implies
(see e.g.~Steutel and van Harn \cite[Proposition 2.2]{Steutel}) that 
the limit $Y$ is compound-Poisson too. That is,
$Y = \sum_{i=1}^N Z_i$, where $Z, Z_1, Z_2, \ldots$ are iid nonnegative 
integer valued random variables, and independently $N$ has Poisson
distribution. This is an important class of distributions, since 
it is exactly the class of infinitely divisible distributions on the nonnegative 
integers, see e.g.~\cite[Theorem II.3.2]{Steutel}.
Interestingly, from the form of the g.f.~$f_Y$ it is difficult to deduce that 
it is compound-Poisson, because the logarithm is given as a 
power series in $1-s$, while for the compound-Poisson this is a power series in $s$. 
Under some conditions on the sequence $(\lambda_n)$ we can rewrite 
$\log f_Y$ to a series expansion in $s$, which allows us to understand the 
structure of the limit.

\begin{theorem}\label{thm:compPoi}
Assume one of the followings: 
\begin{itemize}
\item[(i)] conditions of Theorem \ref{thm:bev1} hold; or
\item[(ii)] conditions of Theorem \ref{thm:bev2} hold and 
$\limsup_{n\to \infty} \lambda_n^{1/n} \leq 1/2$.
\end{itemize}
Then the limiting g.f.~in Theorems \ref{thm:bev1} and \ref{thm:bev2}
can be written as
\[
f_Y(s) = \exp \left\{ \sum_{n=1}^\infty A_n (s^n - 1) \right\},
\]
where
\[
A_n = \frac{\nu^n}{n (2+\nu)^n}
\sum_{j=1}^\infty q_j \left[ \left( 1 + \frac{2}{\nu} \right)^{\min(j,n)} - 1 
\right],
\]
and 
\[
q_j = \lim_{n\to\infty} \frac{\p (\varepsilon_n = j)}{1 - \fl_n}, \quad j=1,2,\ldots.
\]
\end{theorem}

\begin{remark}
In fact, the last limit exists without any extra conditions on the $(\lambda_n)_n$
sequence, see Lemma \ref{lemma:limit}. Furthermore, it is clear that under 
(C4) $q_k = 0$ for $k \geq K$.

The form of $f_Y$ also implies that the limit has the representation
$Y = \sum_{n=1}^\infty n N_n$, where $(N_n)$ are independent Poisson random variables,
such that $N_n$ has parameter $A_n$.
\end{remark}

\section{Proofs} \label{sec:proofs}

\subsection{Preparation}

In the proofs we analyze the g.f.~of the underlying processes. 
To prove distributional convergence for nonnegative integer 
valued processes, it is enough to prove the pointwise convergence of the 
g.f., and show that the limit is a g.f.~as well,
see Feller \cite[p.280]{Feller1}.

Recall that $f_n(s) = \E s^{\xi_n}$ stands for the offspring 
g.f.~in generation $n$. For the composite g.f.~introduce 
the notation $f_{n,n}(s) = s$, for $j < n$ 
\[
f_{j,n}(s) \coloneqq f_{j+1} \circ \ldots \circ  f_n(s),
\]
and for the corresponding means $\fl_{n,n} = 1$,
\[
\fl_{j,n} \coloneqq \fl_{j+1}\ldots \fl_n, \quad j <n.
\]
Then it is well-known that $\E s^{X_n} = f_{0,n}(s)$ and 
$\E X_n = \fl_{0,n}$.

For a g.f.~$f$, with mean $\fl$, define the 
\emph{shape function} as
\begin{equation} \label{eq:def-phi}
\varphi(s) = \frac{1}{1 - f(s)} - \frac{1}{\fl (1-s)},  
\ 0 \leq s < 1, \quad
\varphi(1) = \frac{f''(1)}{2 f'(1)^2}.
\end{equation}
Let $\varphi_j$ be the shape function of $f_j$.
By definition of $f_{j,n}$
\[
\begin{split}
\frac{1}{1 - f_{j,n} (s)} 
& = \frac{1}{\fl_{j+1} (1 
- f_{j+1,n}(s))} + \varphi_{j+1}(f_{j+1,n}(s)),
\end{split}
\]
therefore iteration gives (\cite[Lemma 5]{Kersting}, \cite[Proposition 1.3]{KerVat})
\begin{equation} \label{eq:shape}
\frac{1}{1 - f_{j,n} (s)} = \frac{1}{\fl_{j,n}(1 - s)} + \varphi_{j,n}(s),
\end{equation}
where
\begin{equation} \label{eq:phijn}
\varphi_{j,n}(s)\coloneqq\sum_{k=j+1}^{n} \frac{\varphi_k(f_{k,n}(s))}{\fl_{j,k-1}}.
\end{equation}
The latter formulas show the usefulness of the shape function.
The next statement gives precise upper and lower bounds on the shape function.

\begin{lemma}[{\cite[Lemma 1]{Kersting}, \cite[Proposition 1.4]{KerVat}}]
\label{lemma:shape}
Assume $0<\fl<\infty$, $f''(1)<\infty$ and let $\varphi(s)$ be the shape function of $f$. Then, for $0\leq s\leq 1$,
\[
\frac{1}{2}\varphi(0)\leq \varphi(s)\leq 2\varphi(1).
\]
\end{lemma}

For further properties of shape functions we refer to 
\cite[Chapter 1]{KerVat} and to \cite{Kersting}.

We frequently use the following two simple lemmas from \cite{GyIPV}. The first 
is an extension of the Toeplitz summation lemma, and the second follows from 
the Taylor expansion of the g.f.

\begin{lemma}[{\cite[Lemma 5]{GyIPV}}] \label{lemma:anjk}
Let $(\fl_n)_{n\in\N}$ be a sequence of real numbers satisfying
(C1). Let 
\[
a_{n,j}^{(k)}=(1-\fl_j)\prod_{i=j+1}^{n}\fl_i^k, \quad n,j,k\in\N, \quad j\leq n.
\]
Then, 
$\lim_{n \to \infty} \max_{1\leq j\leq n} a_{n,j}^{(1)}  = 0$, and
for an arbitrary sequence $(x_n)_{n\in\N}$ of real numbers with 
$\lim_{n\to\infty}x_n = x \in \R$, for any $k\in\N$,
\[
\sum_{j=1}^{n} a_{n,j}^{(k)}x_j\to \frac{x}{k}, \quad n \to \infty.
\]
\end{lemma}

\begin{lemma}[{\cite[Lemma 6]{GyIPV}}] \label{lemma:hjs}
Let $\varepsilon$ be a nonnegative integer-valued random variable with 
factorial moments 
\[
m_{k} \coloneqq \E [\varepsilon(\varepsilon-1) \ldots (\varepsilon-k+1)], 
\quad k\in \N,
\]
$m_{0} \coloneqq 1$ and with g.f.~$h(s) = \E s^{\varepsilon}$, $s\in[-1,1]$. 
If $m_{\ell}<\infty$ for some $\ell \in \N$ with $|s|\leq 1$, then
\[
h(s) = \sum_{k=0}^{\ell-1} 
\frac{m_{k}}{k!}(s-1)^k + R_{\ell}(s), \quad |R_{\ell}(s)| \leq 
\frac{m_{\ell}}{\ell!}|s-1|^\ell.
\]
\end{lemma}

\subsection{Proof of Theorem \ref{thm:yaglom-geom}}

\begin{lemma}\label{lemma:phi}
Let $\varphi_n$ be the shape function of $f_n$. Then
under the condition of Theorem \ref{thm:yaglom-geom} with $\nu > 0$
\[
\lim_{n\to\infty}
\sup_{s \in [0,1]}
\frac{|\varphi_n(1)-\varphi_n(s)|}{1-\fl_n} = 0.
\]
\end{lemma}

\begin{proof}
To ease notation we suppress the lower index $n$.
By the Taylor expansion
\[
f(s) = 1+\fl \, (s-1)+\frac{1}{2}f''(1)(s-1)^2+\frac{1}{6}f'''(t)(s-1)^3,
\]
for some $ t\in(s,1)$. Thus, recalling \eqref{eq:def-phi} 
the shape function can be written in the form
\[
\varphi(s) = \frac{\fl \, (1-s)-1+f(s)}{\fl \, (1-s)(1-f(s))} = 
\frac{f''(1)}{2 \fl^2} 
\frac{1 - \frac{f'''(t)}{3 f''(1)}(1-s)}
{1 - \frac{f''(1) }{2 \fl}(1-s) +\frac{f'''(t) }{6 \fl}(1-s)^2}.
\]
Therefore,
\[
\frac{\varphi(1)-\varphi(s)}{1-\fl} = \frac{f''(1)}{2\fl^2(1-\fl)} 
\left(1-\frac{1-\frac{f'''(t)}{3 
f''(1)}(1-s)}{1-\frac{f''(1)}{2\fl}(1-s)+\frac{f'''(t)}{6\fl}(1-s)^2}\right).
\]
Using the assumptions and the monotonicity of $f'''$, 
uniformly in $s \in (0,1]$
\[
\lim_{n \to \infty} 
\left[ 
\frac{f_n'''(t)}{f_n''(1)} 
+ \frac{f_n''(1)}{\fl_n} + \frac{f_n'''(t)}{\fl_n} \right] = 0,
\]
thus the statement follows.
\end{proof}

\begin{lemma}\label{lemma:f-conv}
Under the conditions of Theorem \ref{thm:yaglom-geom}, for $s \in [0,1)$
\[
\lim_{n \to \infty}
\frac{\fl_{0,n}}{1-f_{0,n}(s)} = 
\frac{1}{1-s}+\frac{\nu}{2}.
\]
\end{lemma}

\begin{proof}
Recalling \eqref{eq:shape} and \eqref{eq:phijn} we have to show that
\[
\fl_{0,n} \varphi_{0,n}(s) = \sum_{j=1}^n \fl_{j-1,n} \varphi_j(f_{j,n}(s))
\to \frac{\nu}{2}.
\]
Here, and later on any nonspecified limit relation is meant as $n \to \infty$. 

First let $\nu=0$.
Using Lemmas \ref{lemma:shape} and \ref{lemma:anjk}, 
\[
\begin{split}
\fl_{0,n} \varphi_{0,n}(s) 
& 
\leq  \sum_{j=1}^n \fl_{j,n} \frac{f_j''(1)}{\fl_j} 
\sum_{j=1}^n a_{n,j}^{(1)} 
\frac{f_j''(1)}{\fl_j (1-\fl_j)} \to 0.
\end{split}
\]

For $\nu\in(0,\infty)$ write
\begin{equation} \label{eq:aux-decomp}
\sum_{j=1}^n \fl_{j-1,n} \varphi_j(f_{j,n}(s)) =
\sum_{j=1}^n \fl_{j-1,n} \varphi_j(1) - 
\sum_{j=1}^n \fl_{j-1,n} (\varphi_j(1) - \varphi_j(f_{j,n}(s))).
\end{equation}
By Lemma \ref{lemma:anjk} for the first term we have
\begin{equation} \label{eq:aux1}
\sum_{j=1}^n \fl_{j-1,n} \varphi_j(1) 
= \frac{1}{2} \sum_{j=1}^n a_{n,j}^{(1)} \frac{1}{\fl_j} \frac{f_j''(1)}{1- \fl_j}
\to \frac{\nu}{2}.
\end{equation}
For the second term in \eqref{eq:aux-decomp},
\[
\left| \sum_{j=1}^n \fl_{j-1,n} (\varphi_j(1) - \varphi_j(f_{j,n}(s))) 
\right| \leq 
\sum_{j=1}^{n} a_{n,j}^{(1)} \fl_j
\sup_{s \in [0,1]}
\frac{|\varphi_j(1)-\varphi_j(s)|}{1-\fl_j} \to 0,
\]
according to Lemmas \ref{lemma:anjk} and \ref{lemma:phi}.
Combining with \eqref{eq:aux-decomp} and \eqref{eq:aux1} the statement follows.
\end{proof}

\begin{proof}[Proof of Theorem \ref{thm:yaglom-geom}]
We prove the convergence of the conditional g.f. 
For $s \in (0,1)$ we have by the previous lemma
\[
\begin{split}
\E [ s^{X_n}|X_n>0 ] = \frac{f_{0,n}(s)-f_{0,n}(0)}{1-f_{0,n}(0)} = 1- 
\frac{1-f_{0,n}(s)}{1-f_{0,n}(0)} \to \frac{2}{2+\nu} \frac{s}{1-\frac{\nu}{\nu+2}s},
\end{split}
\]
where the limit is the g.f.~of the geometric distribution with parameter $\frac{2}{2+\nu}$.
\end{proof}

\subsection{Proofs of subsection \ref{ssec:bevandorlas}}

First we need a simple lemma on the alternating sum involving binomial coefficients.

\begin{lemma}\label{lemma:sum}
For $x\in\R$
\[
\sum_{i=1}^{k-1}\binom{k-1}{i}(-1)^{i}\frac{(1+x)^i-1}{i} = 
\sum_{i=1}^{k-1}(-1)^{i}\frac{x^i}{i}.
\]
\end{lemma}

\begin{proof}
Both sides are polynomials of $x$, and equality holds for $x = 0$, therefore it is enough
to show that the derivatives are equal. Differentiating the LHS we obtain
\[
\sum_{i=1}^{k-1} \binom{k-1}{i} (-1)^i (1+x)^{i-1}
= \frac{1}{1+x} \left[ ( 1  - (1+x))^{k-1} - 1 \right].
\]
Differentiating the RHS and multiplying by $(1+x)$, the statement follows.
\end{proof}

\begin{proof}[Proof of Theorem \ref{thm:bev1}]
Recall that $f_n(s)=\E s^{\xi_n}$, and let $g_n(s)=\E s^{Y_n}$ and 
$h_n(s)=\E s^{\varepsilon_n}$. Then the branching property gives that
(see e.g.~\cite[Proposition 1]{GKMP})
\begin{equation*}\label{eq:bev-gfgv}
g_n(s)=\prod_{j=1}^{n}h_j(f_{j,n}(s)).
\end{equation*}
We prove the convergence of the g.f., that is
\begin{equation*} \label{eq:g-conv}
\lim_{n \to \infty } g_n(s) = f_Y(s), \quad s \in [0,1].
\end{equation*}
Fix $s \in [0,1)$. Let us introduce
\begin{equation*} \label{eq:hatg}
\widehat{g}_n(s)=\prod_{j=1}^{n}e^{h_j(f_{j,n}(s))-1}.
\end{equation*}
Note that $\widehat g_n$ is a kind of accompanying law, its distribution
is compound-Poisson, therefore its limit is compound-Poisson too,
see \cite[Proposition 2.2]{Steutel}.
By convexity of the g.f.
\begin{equation}\label{eq:est-fjn}
f_{j,n}(s) \geq 
1+\fl_{j,n}(s-1).
\end{equation}
If $|a_k|, |b_k|< 1$, $k=1,\ldots, n$ then
\begin{equation}\label{eq:prodtosum}
\left|\prod_{k=1}^{n}a_k - \prod_{k=1}^{n}b_k\right|
\leq \sum_{k=1}^{n}\left| a_k-b_k\right|.
\end{equation}
Consequently, using \eqref{eq:prodtosum}, the inequalities
$|e^u-1-u|\leq u^2$ for $|u| \leq 1/2$, and 
$0\leq 1-h_j(s)\leq m_{j,1}(1-s)$, and \eqref{eq:est-fjn}, we have
\begin{equation}\label{eq:gfgv-konv}
\begin{split}
|g_n(s)-\widehat{g}_n(s)|
& \leq \sum_{j=1}^{n}|e^{h_j(f_{j,n}(s))-1}-h_j(f_{j,n}(s))|\\
& \leq \sum_{j=1}^{n}(h_j(f_{j,n}(s))-1)^2\leq 
\sum_{j=1}^{n}\frac{m_{j,1}^2}{1-\fl_j}a_{n,j}^{(2)}\to 0,
\end{split}
\end{equation}
where we used Lemma \ref{lemma:anjk}, 
since ${m_{n,1}}/({1-\fl_n})\to\lambda_1$ implies 
${m_{n,1}^2}/({1-\fl_n})\to 0$. 
Therefore we need to show that
\begin{equation} \label{eq:hatg-conv}
\begin{split}
& \lim_{n\to \infty} \sum_{j=1}^{n}(h_j(f_{j,n}(s))-1) \\
& = -\sum_{k=1}^{K-1}\frac{2^k\lambda_k}{\nu^k}\left(\log\left(1+\frac{\nu}{2}
(1-s)\right) + \sum_{i=1}^{k-1}(-1)^{i}\frac{\nu^i}{i2^i}(1-s)^i\right).
\end{split}
\end{equation}

By Lemma \ref{lemma:hjs},
\begin{equation} \label{eq:hj-series}
h_j(s)= \sum_{k=0}^{K-1} \frac{m_{j,k}}{k!}(s-1)^k + R_{j,K}(s),
\end{equation}
where $|R_{j,K}(s)| \leq \frac{m_{j,K}}{K!}(1-s)^K$, thus
\begin{equation} \label{eq:h-decomp}
\begin{split}
& \sum_{j=1}^{n}(h_j(f_{j,n}(s))-1) 
=\sum_{j=1}^{n} \left[\sum_{k=1}^{K-1} 
\frac{m_{j,k}}{k!}(f_{j,n}(s)-1)^k + R_{j,K}(f_{j,n}(s)) \right] \\
& = \sum_{k=1}^{K-1} {(-1)^k }
\sum_{j=1}^{n} \frac{m_{j,k}}{k! (1 - \fl_j)} a_{n,j}^{(k)}
\left( \frac{1 - f_{j,n}(s)}{\fl_{j,n}} \right)^k 
+ \sum_{j=1}^{n} R_{j,K}(f_{j,n}(s)).
\end{split}
\end{equation}
By \eqref{eq:est-fjn} and Lemma \ref{lemma:anjk},
\begin{equation} \label{eq:R-errror}
\begin{split}
\left|\sum_{j=1}^{n} R_{j,K}(f_{j,n}(s))\right|
& \leq 
\sum_{j=1}^{n} \frac{m_{j,K}}{K!}|f_{j,n}(s)-1|^K
\leq \sum_{j=1}^{n}
\frac{m_{j,K}}{K!}\fl_{j ,n}^K(1-s)^K\\ 
& = (1-s)^K\sum_{j=1}^{n}\frac{m_{j,K}}{K!(1-\fl_j)} a_{n,j}^{(K)}\to 0.
\end{split}
\end{equation}
Moreover, by \eqref{eq:shape},
\begin{equation} \label{eq:fjn-form}
\begin{split}
\frac{\fl_{j,n}}{1-f_{j,n}(s)}
& = \frac{1}{1 - s} +  
\sum_{i=j+1}^{n} a_{n,i}^{(1)} \fl_i \frac{\varphi_i (f_{i,n} (s))}{1- \fl_i} \\
& =: \frac{1}{1-s} + \frac{\nu}{2} (1 - \fl_{j,n}) + \varepsilon_{j,n}.
\end{split}
\end{equation}
By Lemmas \ref{lemma:anjk} and \ref{lemma:phi},
\begin{equation}\label{eq:hiba}
\left|\sum_{i=j+1}^{n} a_{n,i}^{(1)} \fl_i 
\frac{\varphi_i (1)-\varphi_i(f_{i,n}(s))}{1-\fl_i} 
\right|
\leq \sum_{i=1}^n a_{n,i}^{(1)}
\sup_{t \in [0,1] }
\frac{|\varphi_i (1)-\varphi_i(t)|}{1-\fl_i} \to 0,
\end{equation}
and similarly
\begin{equation}\label{eq:sum}
\left| \sum_{i=j+1}^{n} a_{n,i}^{(1)}
\fl_i \frac{\varphi_i(1)}{1-\fl_i} -
\frac{\nu}{2} \sum_{i=j+1}^{n}a_{n,i}^{(1)}\right|
\leq \sum_{i=1}^{n}a_{n,i}^{(1)} \left| \frac{1}{\fl_i} 
\frac{f_i''(1)}{2(1-\fl_i)}-\frac{\nu}{2}\right| \to 0.
\end{equation}
Noting that 
$\sum_{i=j+1}^{n}a_{n,i}^{(1)} = (1-\fl_{j,n})$,
\eqref{eq:hiba} and \eqref{eq:sum} imply
for $\varepsilon_{j,n}$ in \eqref{eq:fjn-form} that
\begin{equation*} \label{eq:vareps-def}
\max_{j\leq n} | \varepsilon_{j,n} | = 
\max_{j\leq n} 
\left|
\frac{\fl_{j,n}}{1-f_{j,n}(s)} - \frac{1}{1-s} - \frac{\nu}{2} ( 1 - \fl_{j,n})
\right| \to 0.
\end{equation*}
The latter further implies that
\[
\limsup_{n \to \infty} \max_{j \leq n} \frac{1- f_{j,n}(s)}{\fl_{j,n}} \leq 1-s,
\]
and that for $n$ large enough by the mean value theorem and \eqref{eq:fjn-form}
\[
\left| \left( \frac{\fl_{j,n}}{1- f_{j,n}(s)} \right)^{-k} - 
\left( \frac{1}{1-s} + \frac{\nu}{2} (1 - \fl_{j,n}) \right)^{-k} \right|
\leq k  \varepsilon_{j,n} .
\]
Thus 
\begin{equation} \label{eq:h-main}
\begin{split}
& \left| 
\sum_{j=1}^{n} a_{n,j}^{(k)} 
\left[ \frac{m_{j,k}}{k! (1 - \fl_j)} 
\left( \frac{1 - f_{j,n}(s)}{\fl_{j,n}} \right)^k  -
 \lambda_k 
\left( \frac{1}{1-s}  + \frac{\nu}{2} { (1 - \fl_{j,n})} \right)^{-k} 
\right] \right| \\
& \leq 
\sum_{j=1}^{n} a_{n,j}^{(k)} 
\Bigg[ 
\left| \frac{m_{j,k}}{k! (1 - \fl_j)} - \lambda_k \right|  
\left( \frac{1 - f_{j,n}(s)}{\fl_{j,n}} \right)^k \\
& \qquad + \lambda_k \left|  
\left( \frac{1 - f_{j,n}(s)}{\fl_{j,n}} \right)^k  -
\left( \frac{1}{1-s}  + \frac{\nu}{2} { (1 - \fl_{j,n})} \right)^{-k} \right|
\Bigg] \\
& \leq \sum_{j=1}^{n} a_{n,j}^{(k)} 
\left| \frac{m_{j,k}}{k! (1 - \fl_j)} - \lambda_k \right|  
+ \lambda_k \sum_{j=1}^{n} a_{n,j}^{(k)}  k \max_{j\leq n} \varepsilon_{j,n}
\to 0,
\end{split}
\end{equation}
where the second inequality holds for $n$ large enough.

Furthermore,
\begin{equation} \label{eq:int-conv}
\lim_{n\to\infty}
\sum_{j=1}^{n} a_{n,j}^{(k)}
\left( \frac{1}{1-s}+\frac{\nu}{2}(1-\fl_{j,n}) \right)^{-k} 
= \int_{0}^{1} 
\frac{y^{k-1}} {\left (\frac{1}{1-s}+\frac{\nu}{2}(1-y)\right )^k} \dd y,
\end{equation}
since the LHS is a Riemann approximation of the integral above
corresponding to the partition $\{\fl_{j,n}\}_{j=1}^n$, and 
$\fl_{j,n} - \fl_{j-1,n} = a_{n,j}^{(1)}\to 0$ uniformly in $j$ 
according to Lemma \ref{lemma:anjk}.
Changing variables $u=\nu (1-s)(1-y)+2$ and using 
the binomial theorem,
\begin{equation} \label{eq:int-conv-2}
\begin{split}
& \int_{0}^{1} \frac{y^{k-1}} {\left (\frac{1}{1-s}+\frac{\nu}{2}(1-y)\right )^k} 
\dd y \\ &
= (-1)^{k+1} \frac{2^k}{\nu^k} 
\bigg( \!\! 
\log\left(1+\frac{\nu}{2}(1-s)\right) + 
\sum_{i=1}^{k-1} \binom{k-1}{i} (-1)^{i} 
\frac{(1+\frac{\nu}{2}(1-s))^{i}-1}{i}  \bigg) \\ &
= (-1)^{k+1} \frac{2^k}{\nu^k}
\left( \log\left(1+\frac{\nu}{2}(1-s)\right) +
\sum_{i=1}^{k-1} (-1)^{i}\frac{\nu^i}{i2^i}(1-s)^i\right),
\end{split}
\end{equation}
where the last equality follows from Lemma \ref{lemma:sum}. 

Substituting back into \eqref{eq:h-decomp},
by \eqref{eq:R-errror}, \eqref{eq:h-main},
\eqref{eq:int-conv}, and \eqref{eq:int-conv-2} we obtain
\eqref{eq:hatg-conv}.
\end{proof}

\begin{proof}[Proof of Theorem \ref{thm:bev2}]
Similarly to the proof of Theorem \ref{thm:bev1}, 
\eqref{eq:gfgv-konv} holds, 
so it is enough to prove that
\begin{equation*} \label{eq:hatg-conv2}
\begin{split}
& \sum_{j=1}^{n}(h_j(f_{j,n}(s))-1) \\
& \to 
-\sum_{k=1}^{\infty}\frac{2^k\lambda_k}{\nu^k} \left(
\log\left(1+\frac{\nu}{2} (1-s)\right)
+\sum_{i=1}^{k-1} (-1)^{i} \frac{\nu^i}{i2^i}(1-s)^i \right).
\end{split}
\end{equation*}
Let $s\in(0,1)$ be a fixed arbitrary number. 
By Taylor's theorem, for 
some $\xi \in (0, \nu(1-s)/2)$
\[
\log\left(1+\frac{\nu}{2} (1-s)\right)
+\sum_{i=1}^{k-1} (-1)^{i} \frac{\nu^i}{i2^i}(1-s)^i  =
\frac{\nu^k (1-s)^k}{2^k} \frac{1}{(1 + \xi)^k} \frac{1}{k}, 
\]
so
\[
\frac{2^k}{\nu^{k}} 
\left| \log\left(1+\frac{\nu}{2} (1-s)\right)
+\sum_{i=1}^{k-1} (-1)^{i} \frac{\nu^i}{i2^i}(1-s)^i 
\right|
\leq \frac{(1-s)^k}{k}.
\]
Therefore, for any $\varepsilon > 0$ there exists $\ell$ 
large enough such that
\begin{equation} \label{eq:inf-sum-1}
\begin{split}
& \left| \sum_{k=\ell}^{\infty} \frac{2^k\lambda_k} {\nu^k} \left(
\log\left( 1 + \frac{\nu}{2} (1-s)\right) 
+\sum_{i=1}^{k-1} (-1)^{i} \frac{\nu^i}{i2^i} (1-s)^i \right) \right| 
\\ 
& \leq \sum_{k=\ell}^\infty \frac{\lambda_k (1-s)^k}{k}
< \varepsilon.
\end{split}
\end{equation}

By Lemma \ref{lemma:hjs}, \eqref{eq:hj-series} holds with $K = \ell$
therefore
\[
\begin{split}
\sum_{j=1}^{n} (h_j(f_{j,n}(s))-1) 
& =\sum_{j=1}^{n} \left[\sum_{k=1}^{\ell-1} 
\frac{m_{j,k}}{k!}(f_{j,n}(s)-1)^k + R_{j,\ell} (f_{j,n}(s)) \right]\\
& = \sum_{k=1}^{\ell-1} \sum_{j=1}^{n} 
\frac{m_{j,k}}{k!} (f_{j,n}(s)-1)^k + \sum_{j=1}^{n} R_{j,\ell}(f_{j,n}(s)).
\end{split}
\]
Moreover, by \eqref{eq:est-fjn} and Lemma \ref{lemma:anjk},
\begin{equation} \label{eq:R2}
\begin{split}
\left| \sum_{j=1}^{n} R_{j,\ell} (f_{j,n}(s)) \right|
&\leq \sum_{j=1}^{n}
\frac{m_{j,\ell}}{\ell!} |f_{j,n}(s)-1|^\ell 
\leq\sum_{j=1}^{n} \frac{m_{j,\ell}}{\ell!} \fl_{j ,n}^\ell (1-s)^\ell\\
& \to \frac{(1-s)^\ell}{\ell}\lambda_\ell \leq \varepsilon.
\end{split}
\end{equation}
Summarizing, 
\[
\begin{split}
& \left| \sum_{j=1}^n ( h_j(f_{j,n}(s)) - 1) - \log f_Y(s) \right| \\
& \leq 
\Bigg| \sum_{k=1}^{\ell - 1}
\sum_{j=1}^{n} \frac{m_{j,k}}{k!} (f_{j,n}(s)-1)^k \\
& \phantom{\leq} \quad + 
\sum_{k=1}^{\ell-1} \frac{2^k\lambda_k}{\nu^k} 
\left(\log\left(1+\frac{\nu}{2}
 (1-s)\right) - \sum_{i=1}^{k-1}(-1)^{i+1}\frac{\nu^i}{i2^i}(1-s)^i \right)
 \Bigg|\\
& \phantom{\leq} \quad + 
\left|\sum_{k=\ell}^{\infty}
\frac{2^k\lambda_k}{\nu^k}\left(\log\left(1+\frac{\nu}{2}
(1-s)\right)-\sum_{i=1}^{k-1}(-1)^{i+1}\frac{\nu^i}{i2^i}(1-s)^i\right)\right| \\
& \phantom{\leq} \quad+ \left|\sum_{j=1}^{n} R_{j,\ell}(f_{j,n}(s))\right|,
\end{split}
\]
where the first term on the RHS converges to 0 by the previous result,
while the second and third are small for $n$ large by \eqref{eq:inf-sum-1} and
\eqref{eq:R2}. As $\varepsilon > 0$ is arbitrary, the result follows.
\end{proof}

\subsection{Proof of Theorem \ref{thm:compPoi}}

Before the proof we need two auxiliary lemmas.

\begin{lemma}\label{lemma:limit}
If condition (C4) of Theorem \ref{thm:bev1} or 
(C4') of Theorem \ref{thm:bev2} hold then 
$q_i = \lim_{n\to\infty}\frac{\p(\varepsilon_n = i)}{1-\fl_n}$ exists
for each $i = 1,2,\ldots$.
\end{lemma}

\begin{proof}
Let $h_n[i] = \p ( \varepsilon_n = i)$.
If (C4) holds, there are only finitely many $\lambda$'s, and 
the statement follows easily by backward induction. However, if 
(C4') holds, a more involved argument is needed, which works in both cases.

The $k$th moment of a random variable can be expressed in terms of 
its factorial moments as
\[
\mu_{n,k} \coloneqq \E \varepsilon_n^k = 
\sum_{i=1}^k \Ssec{k}{i} m_{n,i},
\]
where $\Ssec{k}{i} = \frac{1}{i!} 
\sum_{j=0}^i(-1)^j\binom{i}{j}(i-j)^k$ denotes the Stirling number of the second kind.
On Stirling numbers we refer to Graham et al.~\cite[Section 6.1]{GKP}.
Therefore 
\begin{equation} \label{eq:stirling}
\begin{split}
\sum_{i=1}^\infty \frac{ i^k h_n[i]}{1-\fl_n} 
& = \frac{\mu_{n,k}}{1-\fl_n} = \sum_{i=1}^k 
\Ssec{k}{i} \frac{m_{n,i}}{1-\fl_n} \\
& \to \sum_{i=1}^k 
\Ssec{k}{i} i! \lambda_i \eqqcolon \mu_k.
\end{split}
\end{equation}
Hence the sequence $(h_n[i]/(1- \overline f_n))_{n\in\N}$ 
is bounded in $n$ for all $i$. Therefore any subsequence contains
a further subsequence $(n_\ell)$ such that
\begin{equation} \label{eq:q-def}
\lim_{\ell \to \infty} \frac{h_{n_\ell}[i]}{1 - \overline f_{n_\ell}} = q_i,
\quad \forall i\in\N.
\end{equation}
To prove the statement we have to show that 
the sequence $(q_i)$ is unique, it does not depend on the 
subsequence. Note that the sequence $(q_i)_{i\in\N}$ is not necessarily a 
probability distribution. 

By the Fatou lemma and \eqref{eq:stirling}
\begin{equation} \label{eq:mu-q-Fatou}
\mu_k = \lim_{\ell \to \infty} \frac{\mu_{n_\ell, k}}{1 - \overline f_{n_\ell}}
\geq \sum_{i=1}^\infty \lim_{\ell \to \infty} 
\frac{h_{n_\ell}[i]}{1- \overline f_{n_\ell}} i^k = \sum_{i=1}^\infty q_i i^k.
\end{equation}
Let $\varepsilon > 0$ and $k \in \N$ be arbitrary.
Write $n_\ell = n$, and  
put $K := \left[ \frac{\mu_{k+1}}{\varepsilon}\right] + 1$.
Then
\[
\sum_{i=1}^\infty  \frac{h_n[i]}{1 - \overline f_n} i^{k+1}
\geq \sum_{i=K + 1}^\infty  \frac{h_n[i]}{1 - \overline f_n} 
i^{k}  K,
\]
and hence, by \eqref{eq:stirling} and the definition of $K$
\[
\limsup_{n \to \infty} 
\sum_{i=K + 1}^\infty  \frac{h_n[i]}{1 - \overline f_n} 
i^{k} \leq \varepsilon.
\]
Therefore by \eqref{eq:q-def}
\[
\mu_{k} = 
\lim_{n \to \infty}
\sum_{i=1}^\infty  \frac{h_n[i]}{1 - \overline f_n} 
i^{k} \leq \limsup_{n\to \infty}
\sum_{i=1}^K  \frac{h_n[i]}{1 - \overline f_n} i^k
+ \varepsilon \leq \sum_{i=1}^\infty q_i i^k + \varepsilon.
\]
Since $\varepsilon > 0$ is arbitrary by equation \eqref{eq:mu-q-Fatou}
\begin{equation} \label{eq:mu-q}
\mu_{k} = \sum_{i=1}^\infty q_i i^k.
\end{equation}
Using the Stieltjes moment problem we show that the sequence 
$(\mu_k)$ uniquely determines $(q_i)$.
In order to do that, it is enough to show that
Carleman's condition (\cite [Theorem 5.6.6.]{Simon}) is fulfilled, that 
is, $\mu_k$ is not too large. 
Since $\limsup_{n\to \infty} \lambda_n^{1/n} \leq 1$, for $n$ large 
$\lambda_n \leq 2^n$. Furthermore, by trivial upper bounds
\[
\Ssec{k}{i} i! = \sum_{\ell=0}^i (-1)^\ell \binom{i}{\ell} (i - \ell)^k
\leq i^k 2^i,
\]
thus, by \eqref{eq:stirling} for some $C >0$
\[
\mu_k = \sum_{i=0}^k \Ssec{k}{i} i! \lambda_i \leq 
C + \sum_{i=0}^k i^k 4^i \leq C (4 k )^k \leq k^{2k},
\]
for $k$ large enough, showing that Carleman's condition holds.
Hence the sequence $(q_i)_{i\in\N}$ is indeed unique, and the proof is complete.
\end{proof}

\begin{lemma} \label{lemma:binom-3}
For any $L \geq 1$, $n \geq 1$, and $x \in \R$
\[
\sum_{j=1}^L \sum_{\ell = 1}^j (-1)^{\ell +j} 
\binom{L+n}{j+n} \binom{j-\ell + n - 1}{n-1} x^\ell = (1+x)^L -1.
\]
\end{lemma}

\begin{proof}
Changing the order of summation, it is enough to prove that
for $L\geq 1$, $n \geq 1$, $0 \leq \ell \leq L$
\begin{equation*} \label{eq:binom-3}
\sum_{j=\ell}^L (-1)^{j+\ell} \binom{L+n}{j+n} \binom{j-\ell+n-1}{n-1}
= \binom{L}{\ell}.
\end{equation*}
We prove by induction on $L$.
This holds for $L = 1$, and assuming for $L \geq 1$, for $L+1$
we have by the induction hypothesis for $(L,\ell, n)$ and $(L, \ell-1, n)$
\[
\begin{split}
& \sum_{j=\ell}^{L+1} (-1)^{j+\ell} \binom{L+1+n}{j+n} 
\binom{j-\ell+n-1}{n-1}   \\ 
& = \binom{L}{\ell} + 
\sum_{j=\ell}^{L+1} (-1)^{j+\ell} \binom{L+n}{j+n-1} \binom{j-\ell+n-1}{n-1} \\
& = \binom{L}{\ell} + 
\sum_{j=\ell-1}^{L} (-1)^{j+\ell-1} \binom{L+n}{j+n} 
\binom{j-(\ell-1)+n-1}{n-1} \\
& = \binom{L}{\ell} + \binom{L}{\ell-1} = \binom{L+1}{\ell},
\end{split}
\]
as claimed.
\end{proof}

\begin{proof}[Proof of Theorem \ref{thm:compPoi}]
In order to handle assumptions (i) and (ii) together, 
under (i) we use the notation $\lambda_k = 0$ for $k \geq K$. Then 
the limiting g.f.~$f_Y$ is given in \eqref{eq:Y-genfc}.

By the Taylor series of the logarithm
\begin{equation*} \label{eq:log-series}
\log\left(1+\frac{\nu}{2}(1-s)\right) = \log\left(1+\frac{\nu}{2}\right) -  \sum_{n=1}^\infty n^{-1} 
\left(\frac{\nu}{2+\nu}\right)^{n} s^n, \quad s \in (0,1),
\end{equation*} 
and expanding $(1-s)^i$ in \eqref{eq:Y-genfc}
we have
\begin{equation} \label{eq:fY-series1}
\begin{split}
& f_Y(s) = \exp \bigg\{  
- \sum_{k=1}^\infty \frac{2^k \lambda_k}{\nu^k} 
\left(
 \log\left(1+\frac{\nu}{2}\right) + \sum_{i=1}^{k-1}(-1)^{i} 
\frac{\nu^{i}}{i2^{i}} 
\right) \\ 
& + \sum_{k=1}^\infty \frac{2^k \lambda_k}{\nu^k} \sum_{n=1}^\infty 
s^n \left( 
n^{-1} \left(\frac{\nu}{2+\nu}\right)^{n} - 
\ind_{n \leq k-1} \sum_{i=n}^{k-1} (-1)^{i+n} \frac{\nu^i}{i 2^i}
\binom{i}{n} \right)
\bigg\}.
\end{split}
\end{equation}
We claim that the order of summation in \eqref{eq:fY-series1} 
with respect to $k$ and $n$ can be interchanged.  Indeed, this 
is clear under (i), since the sum in $k$ is in fact finite.

Assume (ii). Then 
Taylor's theorem applied to the function $(1+x)^{-n}$ gives
\[
(1 + x)^{-n} = 
n \sum_{\ell=n}^{k-1} (-1)^{\ell +n } \binom{\ell}{n} 
\frac{x^{\ell-n}}{\ell} + (-1)^{k-n} \binom{k-1}{k-n} 
(1 + \xi)^{-k} x^{k-n},
\]
with $\xi \in (0,x)$. Therefore,
\begin{equation} \label{eq:cpoi-aux1}
\begin{split}
\left| 
\sum_{i=n}^{k-1} (-1)^{i+n} \binom{i}{n} \frac{x^{i-n}}{i} - 
n^{-1} (1 + x)^{-n} \right| 
\leq n^{-1} \binom{k-1}{k-n} x^{k-n},
\end{split}
\end{equation}
thus with $x = \nu / 2$
\[
\begin{split}
& \sum_{k=1}^\infty \frac{2^k \lambda_k}{\nu^k}
\sum_{n=1}^{k-1} s^n \left| 
\sum_{i=n}^{k-1} (-1)^{i+n} \binom{i}{n} \frac{\nu^{i}}{2^i i} - 
n^{-1} \left( \frac{\nu}{2 + \nu} \right)^{n} \right| \\
& \leq \sum_{k=1}^\infty \frac{2^k \lambda_k}{\nu^k} 
\sum_{n=1}^{k-1} s^n \frac{\nu^n}{2^n}  n^{-1} \binom{k-1}{k-n} 
\frac{\nu^{k-n}}{2^{k-n}} \\
& = \sum_{k=1}^\infty \lambda_k 
\sum_{n=1}^{k-1} s^n n^{-1} \binom{k-1}{k-n} \\
& \leq \sum_{k=1}^\infty \lambda_k ( 1 +s)^k,
\end{split}
\]
which is finite for $s \in (0,1)$ if $\limsup \lambda_n^{1/n} \leq 1/2$.
The other part is easy to handle, as
\[
\begin{split}
\sum_{k=1}^\infty \frac{2^k \lambda_k}{\nu^k} 
\sum_{n=k}^\infty s^n n^{-1} \left( \frac{\nu}{2 + \nu}\right)^n 
& \leq \sum_{k=1}^\infty \frac{2^k \lambda_k}{\nu^k}
s^k \left( \frac{\nu}{2 + \nu}\right)^k \\
& = \sum_{k=1}^\infty \frac{2^k \lambda_k}{(2 + \nu)^k} s^k,
\end{split}
\]
which is summable.

Summarizing, in both cases the order of summation in 
\eqref{eq:fY-series1} can be interchanged, and doing so we 
obtain
\begin{equation*} \label{eq:f-series}
f_Y(s) = \exp \left\{ -A_0 + \sum_{n=1}^\infty A_n s^n \right\},
\end{equation*}
where
\begin{equation} \label{eq:A-def}
\begin{split}
A_0 & = \sum_{k=1}^{\infty} \frac{2^k\lambda_k}{\nu^k}
\left(\sum_{i=1}^{k-1}(-1)^{i} 
\frac{\nu^{i}}{i2^{i}} + \log\left(1+\frac{\nu}{2}\right)\right) \\
A_n & = 
\sum_{k=1}^{\infty}
\frac{2^k\lambda_k}{\nu^k} 
\left( 
\left(\frac{\nu}{2+\nu}\right)^n n^{-1} - 
\ind_{n \leq k-1}
\sum_{i=n}^{k-1}
(-1)^{i+n}\binom{i}{n}\frac{\nu^{i}}{i2^{i}}\right).
\end{split}
\end{equation}

Recall the notation from the proof of Lemma \ref{lemma:limit}. 
By \eqref{eq:stirling}, using the inversion formula for Stirling numbers of the 
first and second kind we have, see e.g.~\cite[Exercise 12, p.~310]{GKP},
\[
k! \lambda_k = \sum_{i=0}^k (-1)^{k+i} \Sfirst{k}{i} \mu_k,
\]
where $\Sfirst{k}{i}$ stands for the Stirling number of the first kind. 
Substituting \eqref{eq:mu-q}, and using that 
\[
\sum_{i=0}^k (-1)^{k+i} \Sfirst{k}{i}  j^k = j (j-1) \ldots (j-k+1)
\]
see, e.g.~\cite[p.263, (6.13)]{GKP}, we obtain the formula
\begin{equation} \label{eq:lambda-q}
\lambda_k = \sum_{j=k}^\infty \binom{j}{k} q_j.
\end{equation}
We also see that $\lambda_k = 0$ implies $q_j = 0$ for all $j \geq k$.
Substituting \eqref{eq:lambda-q} back into \eqref{eq:A-def}
we claim that the order of summation with respect to $k$ and $j$ 
can be interchanged. This is again clear under (i), while under 
(ii), by \eqref{eq:cpoi-aux1}
\[
\begin{split}
& \sum_{k=n+1}^\infty \sum_{j=k}^\infty \frac{2^k }{\nu^k} q_j 
\binom{j}{k}
\left| \sum_{i=n}^{k-1} (-1)^{i+n} \binom{i}{n} \frac{\nu^i}{i 2^i} 
- n^{-1} \frac{\nu^n}{(2+\nu)^n} \right| \\
& \leq \sum_{k=n+1}^\infty \sum_{j=k}^\infty 
\frac{2^k }{\nu^k} q_j  \binom{j}{k}
\frac{\nu^n}{2^n} n^{-1} \binom{k-1}{k-n} \frac{\nu^{k-n}}{2^{k-n}} \\
& = \sum_{k=n+1}^\infty \lambda_k n^{-1} \binom{k-1}{k-n} < \infty,
\end{split}
\]
since
\[
n^{-1} \binom{k-1}{k-n} = k^{-1} \binom{k}{n} \leq k^{n}.
\]
Thus the order of the summation is indeed interchangeable, 
and we obtain 
\begin{equation} \label{eq:B-def}
\begin{split}
A_n & = \sum_{j=1}^\infty q_j 
\bigg\{ 
\bigg[ \left( 1 + \frac{2}{\nu} \right)^j - 1 \bigg]
\left( \frac{\nu}{2 + \nu} \right)^n n^{-1} \\
& \qquad - \ind_{n \leq j-1}
\sum_{k=n+1}^j \frac{2^k}{\nu^k} \binom{j}{k}
\sum_{i=n}^{k-1} (-1)^{i+n} \binom{i}{n} \frac{\nu^i}{i 2^i} 
\bigg\} \\
& = : \sum_{j=1}^\infty q_j B(n,j).
\end{split}
\end{equation}
We claim that 
\begin{equation} \label{eq:B-form}
B(n,j) = \frac{\nu^n}{(2+\nu)^n n } \left[ \left( 1 + \frac{2}{\nu} \right)^{n \wedge j} -1
\right].
\end{equation}
This is clear for $j \leq n$. Writing $\ell = k -i$ in the summation
and using Lemma \ref{lemma:binom-3} with $L = j-n$ we obtain
\[
\begin{split}
& \sum_{k= n +1}^j \sum_{i=n}^{k-1} (-1)^i \binom{j}{k} 
\binom{i-1}{n-1} x^{k-i} \\
& = \sum_{k=n+1}^j \sum_{\ell = 1}^{k-n} (-1)^{k+\ell} 
\binom{j}{k} \binom{k-\ell-1}{n-1} x^\ell \\
& = (-1)^n \left[ (1+x)^{j-n} -1 \right].
\end{split}
\]
Using that $n \binom{i}{n}  = i \binom{i-1}{n-1}$, and
substituting back into \eqref{eq:B-def} with $x= 2/\nu$ we have
\[
B(n,j) = n^{-1} \left[ 1 -
\left( 1 + \frac{2}{\nu} \right)^{j-n} \right] 
+ \left[ \left( 1 + \frac{2}{\nu} \right)^j - 1 \right]
\left( \frac{\nu}{2 + \nu} \right)^n n^{-1},
\]
which after simplification gives \eqref{eq:B-form}.
\end{proof}

\noindent \textbf{Acknowledgements.}
We are grateful to M\'aty\'as Barczy for useful comments and suggestions.
This research was supported by the Ministry of Innovation and
Technology of Hungary from the National Research, Development
and Innovation Fund, project no.~TKP2021-NVA-09. Kevei's
research was partially supported by the J\'anos Bolyai Research 
Scholarship of the Hungarian Academy of Sciences.

\end{document}